\documentclass[a4paper,reqno,12pt]{article}

\usepackage{color}

\usepackage[T1]{fontenc}
\usepackage[utf8]{inputenc}
\usepackage{stix2}
\usepackage{microtype}

\usepackage[left=2.5cm,right=2.5cm,top=2cm,bottom=2.5cm]{geometry}
\usepackage{cite}

\usepackage{authblk}

\usepackage[labelfont={bf},%
                labelsep=period,%
                justification=raggedright,%
                singlelinecheck=false,%
                font=small]
                {caption}
                
\usepackage{amsthm,mathtools}

\theoremstyle{plain}
\newtheorem{theorem}{Theorem}[section]
\newtheorem{corollary}[theorem]{Corollary}
\newtheorem{prop}[theorem]{Proposition}
\newtheorem{lemma}[theorem]{Lemma}
\theoremstyle{definition}
\newtheorem{definition}[theorem]{Definition}
\newtheorem{remark}[theorem]{Remark}

\newcommand{\Rb}{\ensuremath{\mathbb R}}
\newcommand{\Cc}{\ensuremath{\mathcal C}}
\newcommand{\Mc}{\ensuremath{\mathcal M}}

\def\veps{\varepsilon}
\def\vphi{\varphi}

\begin{document}

\title{Long-time asymptotic of the Lifshitz--Slyozov equation with nucleation}
\author[$^*$]{Juan Calvo}
\affil[$^*$]{\footnotesize Universidad de Granada, Avda. Hospicio s/n, 18071, Granada, Spain.  Research Unit ``Modeling Nature'' (MNat). E-mail: juancalvo@ugr.es}
\author[$\dag$]{Erwan Hingant}
\affil[$\dag$]{\footnotesize LAMFA, Université de Picardie Jules Verne, CNRS, 33 rue Saint-Leu, 80039 Amiens, France. E-mail: erwan.hingant@u-picardie.fr}
\author[$\ddag$]{Romain Yvinec}
\affil[$\ddag$]{\footnotesize Université Paris-Saclay, Inria, Centre Inria de Saclay, 91120, Palaiseau, France \& PRC, INRAE, CNRS, Université de Tours, 37380, Nouzilly, France. E-mail: romain.yvinec@inrae.fr}

\date{\normalsize \today}

\maketitle

\begin{abstract} 
We consider the Lifshitz--Slyozov model with inflow boundary conditions of nucleation type. We show that for a collection of representative rate functions the size distributions approach degenerate states concentrated at zero size for sufficiently large times. The proof relies  on monotonicity properties of some quantities associated to an entropy functional. Moreover, we give numerical evidence on the fact that the convergence rate to the goal state is algebraic in time. Besides their mathematical interest, these results can be relevant for the interpretation of experimental data.

\medskip



\end{abstract}

\maketitle
\section{Introduction}
\label{sec:intro}

In this work we study the long time behavior of the Lifshitz--Slyozov (LS) model with nucleation boundary conditions. The model reads:

\begin{equation} 
\label{eq:LS-N}
 \frac{\partial f(t,x)}{\partial t} + \frac{\partial \{(a(x)u(t)-b(x))f(t,x)\}}{\partial x} = 0  \vphantom{\int} \, , \quad t>0\,, \ x\in(0,\infty)\,,
 \end{equation}
 with
\begin{equation}
\label{eq:LS-con}
 u(t) + \int_0^\infty xf(t,x)\, dx = \rho \, , \quad t>0\\[0.8em]
\end{equation}
for some given $\rho>0$, subject to the boundary condition 
\begin{equation}
\label{eq:bc}
((a(x)u(t) - b(x))f(t,x))_{|x=0} = \mathfrak n (u(t)) \,, \quad t \in \{ s>0 \, : \, u(s) > \Phi_0\}
\end{equation}
and the initial condition
\begin{equation}
\label{eq:LS-Nini}
f(0,x)=f^{in}(x)\,, \qquad x\in (0,\infty)\:.
\end{equation}

The LS model describes the temporal evolution of a mixture of monomers and aggregates that undergo the following interactions: a monomer can join an existing aggregate of size $x$, with an attachment rate $a(x)$, and a monomer can detach from an existing aggregate of size $x$, with a detachment rate $b(x)$. The variable $x$ describes the size of the aggregates, so that $f(t,x)$ is the number density of aggregates at time $t$, whereas $u(t)$ stands for monomer concentration at time $t$. Equation \eqref{eq:LS-con} simply encodes the fact that the total mass $\rho$ is preserved. 

Depending on the specific rates $a(x)$ and $b(x)$, the model may or may not need a boundary condition. No boundary condition is needed for the original Lifshitz--Slyozov version of the model \cite{Lifshitz61}, and this is also the case for the various instances of ``Ostwald ripening'' that have been analyzed in the literature, e.g. \cite{Collet00,Collet02a,Collet04,Laurencot01,Niethammer00,Niethammer03,Niethammer05}. In this article we are interested in situations where the kinetic rates are such that a boundary condition is needed to make sense of the model. Classical examples are power-law rates like $a(x) \propto x^\alpha$ and $b(x)\propto x^\beta$ with $0\le \alpha \le \beta \le 1$. For the main applications we have in mind the nucleation rate follows a mass action kinetics, that is $\mathfrak n(u) \propto u^{i_0}$ with $i_0\in \mathbb{N}^*$.

The boundary condition \eqref{eq:bc} encodes the creation of new aggregates from the available pool of monomers. We understand this as an effective description of a nucleation process. 
A boundary condition of the form \eqref{eq:bc} has been deduced from appropriate scaling limits in the case of a second order nucleation kinetics, see \cite{Collet02,Deschamps17} for details. Higher order mass- action-kinetics can also be used to describe the nucleation process in a  phenomenological way \cite{Prigent12,Xue08,Bishop84}; such boundary conditions arise as scaling limits of modified versions of the Becker--D\"oring model. It is also interesting to note that some nucleation boundary conditions were derived in \cite{Deschamps17} that do not follow a mass action law kinetics. This can be relevant for the description of protein polymerization phenomena. All in all, this is a first step towards a number of important applications in the science of materials and in the field of neurodegenerative diseases \cite{Prigent12,Ross04,Ross05,Szavits-Nossan14}.

We prove below that under generic conditions, solutions will concentrate at vanishing aggregate sizes, while the concentration of available monomers drops down to the activation threshold $\Phi_0$ in \eqref{eq:bc}. Under more specific hypotheses we can get more precise information about the temporal rates at which this dynamics takes place. We complement this with a numerical investigation. The general picture that emerges points to the fact that the concentration dynamics is quite slow, actually taking place with algebraic rates (that depend on $\alpha, \beta$ and $\mathfrak n$). These results are in line with previously known results for the case of outgoing characteristics, which indicate that degenerate steady states are approached at an algebraic rate and with a universal profile, see \cite{Carrillo04,Collet02a}. 

Since the transient behavior for \eqref{eq:LS-N}--\eqref{eq:bc} spans a wide temporal scale, this model is thus  suitable for comparison with experimental data, e.g. those originating in protein polymerization experiments {\em in vitro}. We conjecture the asymptotic profile, after a suitable normalisation, to be universal. However, if we are interested in very long time scales the concentration behavior will take over and then it seems advisable to introduce corrections to the model in order to get a more realistic goal state, see for instance \cite{Calvo18,Collet02,Conlon10,Hariz99,Marder87,Meerson99,Velazquez98}.

\section{Statement of the problem and results}

Here below $\Rb_{+} = (0,+\infty)$ and $(1+x)dx$ denotes the measure with density $x\mapsto (1+x)$ with respect to the Lebesgue measure on $\Rb_{+}$. The space $L^{1}(\Rb_{+},(1+x)dx)$ denotes the space of integrable function w.r.t. the former measure. This space might be endowed with the weak topology denoted by $w$ whose convergence is characterized against bounded functions. Finally, $\Cc([0,\infty),w-L^1(\Rb_+,(1+x)dx))$ is the set of continuous functions from $[0,+\infty)$ into $L^{1}(\Rb_+,(1+x)dx)$ equipped with its weak topology that is, for such $f$,
\[ t\mapsto \int_{0}^{\infty} f(t,x) \vphi(x) (1+x)dx \]
is continuous on $[0,\infty)$ for all $\vphi\in L^{\infty}(0,\infty)$. Here and in the sequel $L^{p}$ refers to the standard Lebesgue spaces.
{We may add a subscript to it, so that the integration variable is made clear.}   
$\Cc^k$ is the space of continuous function whose $k^{th}$ derivatives are continuous, and $\Cc_c^k$ stands for its subspace consisting of compactly supported functions.  

In the remainder the rates $a$, $b$ and $\mathfrak n$ are assumed nonnegative and continuous functions on $[0,\infty)$ and such that
\begin{equation}
\label{H0} \tag{H0}
\forall x>0\,, \ a(x) > 0 \quad \text{ and }\quad \Phi(x) := \frac{b(x)}{a(x)} \to_{x\to 0^{+}} \Phi_{0} \,.
\end{equation}

We deal with global solutions for which the boundary condition is defined for all time, namely
solutions defined in the following sense:

\begin{definition}
  \label{def:sol}
  Let $0\le f^{\rm in}\in L^1(\Rb_+,(1+x)dx)$ and $\rho>0$. We say
  that $0\le f \in \Cc([0,\infty),w-L^1(\Rb_+,(1+x)dx))$ is a (global) solution to the LS equation
  with nucleation \eqref{eq:LS-N}-\eqref{eq:LS-Nini} provided that:
\begin{enumerate}
\item For all $t\geq 0$ and for every $\varphi\in \mathcal C^0([0,\infty))$ such that
  $\varphi'\in L^\infty(0,\infty)$, we have
  \begin{multline} \label{eq:LS-moment-equation}
    \int_0^\infty \varphi(x)f(t,x)\, dx =  \int_0^\infty \varphi(x)f^{\rm in}(x)\, dx 
    + \int_{0}^{t} \int_0^\infty (a(x)u(s)-b(x))\varphi'(x)f(s,x)\, dx \, ds \\
    + \int_{0}^{t} \varphi(0) \mathfrak n(u(s)) \, dt\,.
  \end{multline}
  
\item For all $t\geq 0$,
  \[u(t) \coloneqq  \rho - \int_0^\infty xf(t,x) dx > \Phi_{0}. \]
 \end{enumerate} 
\end{definition}

Note that for any smooth test function $\vphi$, a solution as defined above satisfies the following
moment equation:
\begin{equation}\label{eq:moment_eq_general}
\frac{d}{dt}\int_0^\infty \vphi(x) f(t,x) =  \vphi(0)\mathfrak n(u(t)) + \int_0^\infty (a(x)u(t) - b(x))\vphi'(x)f(t,x)\,.
\end{equation}

Our main hypothesis relies on the monotonicity of $\Phi$. Namely, we suppose that
\begin{equation}
\tag{H1}
\label{H1}
  \Phi \text{ is strictly increasing} \,.
\end{equation}
We shall assume some technical hypotheses as well: First, for any $\veps>0$,
\begin{equation}
\label{H:ab_deriv} \tag{H2}
a' \,, \ b'  \in L^{\infty}(\veps,\infty)\,, 
\end{equation}
which entails the existence of a constant $C>0$ such that, for all $x\geq 0$,
\begin{equation}
\label{eq:sublinear-rate}
a(x)+b(x) \leq C(1+x)\,.
\end{equation}
Next, we suppose
\begin{equation}
  \label{H:a} \tag{H3}
  \inf_{x\in (1,\infty)} a(x) > 0 \,, \text{ and } \frac 1 a \in L^{1}(0,1)\,,
\end{equation}
and that there exists a constant $C>0$ such that, for all $x\geq 0$,
\begin{equation}
\tag{H4}
\label{eq:bphi}
  b(x)\Phi(x) \geq \tfrac 1 C \min(1,x^2)\,.
\end{equation}
Concerning the nucleation rate, we assume that there exist two constants $c>0$ and $k_0\geq 1$ such
that for all $z\geq 0$
\begin{equation}
\tag{H5}
\label{H:nucleation}
\mathfrak n(z) \geq  c z^{k_0} \,.
\end{equation}
Finally, we will assume that the initial condition $f^{\rm in}$ belongs to $L^1(\Rb_+,(1+x)dx)$, is
nonnegative and moreover
\begin{equation}
\tag{H6}
\label{H:init}
  \int_0^\infty  \left(\int_0^x\Phi(z)dz + x^2 \right)f^{\rm in}(x)dx < \infty\:.
\end{equation}

The well-posedness of \eqref{eq:LS-N}-\eqref{eq:LS-Nini} was studied in \cite{Calvo21} under some
additional conditions of a technical nature -which are probably non-optimal. At any rate, taking
\eqref{H1} for granted we know for sure that solutions (should they exist) will be global in
time with $u(t)>\Phi_0$ for all times. Therefore, during the rest of the document we will only concentrate on the set of assumptions
needed for our analysis of the long time behavior.

Some comments are in order concerning our set of hypotheses. Besides ensuring well-posedness,
Hypothesis \eqref{H1} is the natural counterpart to the typical Ostwald Ripening phenomena
where, $\Phi(x) = x^{-1/3}$ for which the long-time asymptotics have been studied
e.g. \cite{Collet02a}. Note that \eqref{H:ab_deriv} is not that demanding; actually, blow-up
phenomena may take place for strictly superlinear rates \cite{Collet00}. Hypothesis \eqref{H:a}, has
been required in connection with the 
well-posedness of the problem, so that characteristics go back to $x=0$ in
finite time and 
render the boundary condition relevant. Condition \eqref{eq:bphi} is purely technical
and does not imply a 
strong restriction, recall the 
power law rates introduced in Section \ref{sec:intro}, namely $a(x) \propto x^\alpha$ and $b(x)\propto x^\beta$ with $0\le \alpha \le \beta \le 1$. Indeed,
such rates satisfy all our hypotheses.  Concerning the nucleation rate, any mass action
kinetics can be considered under \eqref{H:nucleation}. Finally, \eqref{H:init} imposes just a mild
technical requirement on the initial condition.

In order to ascertain the temporal evolution of the system, the number of aggregates constitutes an
important quantity that is defined as
\[M_0(t):=\int_0^\infty f(t,x)\, dx\,.\]
According to the boundary condition \eqref{eq:bc}, this evolves in time via
\[\frac{d M_0}{dt}=\mathfrak n(u)\,.\]

To discuss concentration phenomena, we define $\Mc^+_\rho([0,+\infty))$ the set of nonnegative Radon
measure on $[0,+\infty)$ with total variation less or equal to $\rho$. This space can be equipped
with the weak topology whose convergence corresponds to the convergence against any
continuous and bounded function.

Our main results in this document describe concentration phenomena for the long-time evolution of LS
equation with nucleation boundary conditions:
\begin{theorem}
\label{thm:largetime}
Under hypothesis \textup{(\ref{H0}-\ref{H:init})}, any global solution in the sense of Definition
\ref{def:sol} satisfies
\begin{itemize}
\item $\lim_{t\to +\infty} M_0(t)=+\infty$,
\item $\lim_{t\to +\infty} u(t)=\Phi_0$,
\item $\lim_{t\to +\infty} xf(t,x)dx = (\rho-\Phi_0) \delta_0$, weakly in $\Mc^+_\rho([0,\infty))$.
\end{itemize} 
\end{theorem}

The proof relies on a Lyapunov functional, which we introduce in Section \ref{sec:lyapou}. Right
after that we proceed with the proof of Theorem \ref{thm:largetime} in Section \ref{sec:dust}.

Note that Theorem \ref{thm:largetime} is a generalisation of an earlier result proved in
\cite{Calvo18} for $a(x)=1$ and $b(x)$ such that $0<c_1\le b'(x)\le c_2$ for some constants $c_1$
and $c_2$.  In that particular case, the rates can be computed explicitly. Those are algebraic and
depend on the specific form of the nucleation rate, but are nevertheless quite slow. For more general coefficients, our method of proof cannot provide specific estimates and therefore we cannot ascertain the timescales over which the average aggregate size tends to zero. However, we performed numerical simulations,
whose results suggest that this algebraic trend is actually what we should expect generically. This
is discussed in Section \ref{sec:numeric}. Our results are complemented with an analysis of the case
of constant $\Phi$, which requires a separate treatment. Its long time behavior is analyzed in
Section \ref{sec:const} and expands on the results given in \cite{Collet04}.

One of the main points underlying the previous results is to discriminate whether the system will be
able to fuel nucleation reactions to the extent that the number of fragments grows without
control. We actually show that this is the case for a representative number of situations. Since the
total mass is preserved, this suggests that the average aggregate size becomes smaller and smaller,
which is an instance of {\em dust formation}.

\section{Lyapunov functional}
\label{sec:lyapou}

We shall introduce a Lyapunov functional in the vein of \cite{Collet02a,Calvo18}. For $k$ a
continuous and positive function on $[0,+\infty)$ with continuous derivative and $f$ a solution of \eqref{eq:LS-N}--\eqref{eq:LS-Nini}, we
define for all $t\geq 0$
\begin{equation}\label{eq:H_k}
    H_k(t)=\int_0^\infty k(x)f(t,x)dx + K(u(t))\,,
\end{equation}
with $K(v) = \int_0^v k'\circ \Phi^{-1}(z)dz$. This makes sense since $\Phi$ is monotonous. The
functional $H_k$ is a Lyapunov functional and its time derivative $D_k$ is called its dissipation,
as the following result makes clear.
  
\begin{prop}
\label{prop:lyap}
Assume $0\leq k\in \Cc^{1}([0,+\infty))$ is convex with $k(0)=0$ and $f$ is a solution in the sense of Definition \ref{def:sol}. If
$H_k(0) < \infty$, then $t\mapsto H(t)$ is non-increasing, non-negative and for all $t\geq 0$
\begin{equation} \label{eq:dotH_k}
    H_k(t) + \int_s^t D_k(s) ds  \leq H_k(s)\,,\qquad \forall \, 0\leq s<t  
\end{equation}
where
\[ 0 \leq D_k(t) = \int_0^\infty \left(K'(u(t)) -K'(\Phi(x)) )\right)\left(u(t)-\Phi(x)\right)a(x)f(t,x)dx 
\]
belongs to $L^1_{t}(0,\infty)$. 
\end{prop}

\begin{proof}
  Let $R>0$, $k_R(x) = k(x)$ for $x<R$ and $k_R(x) = k'(R)(x-R)+k(R)$ for $x\geq R$. Notice that
  $k_R$ can be used as a test function in \eqref{eq:LS-moment-equation}. Moreover, $k_R$ is
  convex. We construct $H_{k_R}$ via formula \eqref{eq:H_k} and we compute the dissipative part,
  which is nonnegative because $K'=k_R'\circ\Phi^{-1}$ is increasing. Then we conclude by Fatou's
  lemma.
\end{proof}

As a straightforward consequence with $k(x)=\tfrac 1 2 x^2$ we get the following result:
\begin{corollary}
\label{cor:lyap} Assume \eqref{H:init} and let $f$ be a solution in the sense of Definition \ref{def:sol}. We define
  \begin{equation}
    \label{eq:H_2}
    H(t)= \frac 1 2 \int_0^\infty x^2 f(t,x)dx + \Psi(u(t))\,,
  \end{equation}
  with $\Psi(v) = \int_0^v \Phi^{-1}(z)dz$. We have that $H(0)<\infty$ and $t\mapsto H(t)$ is
  non-increasing, non-negative and the dissipation part is given, for all $t\geq 0$, by
  \[ 
  D(t) =  \int_0^\infty \left(\Phi^{-1}(u(t))- \Phi^{-1}(\Phi(x))\right) \left(u(t)-\Phi(x)\right) a(x)f(t,x)dx\,. 
  \]
\end{corollary}

\begin{remark}
\label{rm:choices}
There are other useful choices of the function $k$.  We will make use of $k(x)=\int_0^x \Phi(y)dy$ which
leads to $k'\circ \Phi^{-1}=Id$ and then
\begin{equation}
\frac{d}{dt}H_k \leq  -\int_0^\infty (u(t)-\Phi(x))^2a(x)f(t,x)dx \in L^1_t(0,\infty).
\end{equation}
We may also take $k(x)=x^{\eta}$ with $\eta\geq 1$ to control moments of the form
\[M_{\eta}(t) := \int_{0}^{\infty} x^{\eta}f(t,x) dx\,.\] Indeed, for any moment
greater than $1$ which is initially finite, we have a uniform bound for every positive time.
\end{remark}

\section{Proof of the main result}
\label{sec:dust}

\subsection{The number of fragments diverges}

We prove first a generic result showing that a shattering phenomenon takes place on long time
intervals: the number of fragments diverges.

\begin{prop}\label{prop:m0}
Assume \textup{(\ref{H0}, \ref{H:ab_deriv},\ref{eq:bphi}-\ref{H:init})}. Then, for every solution in the sense of Definition \ref{def:sol}, there holds that $\lim_{t\to+\infty} M_0(t) = +\infty.$

\end{prop}

\begin{proof}
Taking $\mathbf 1$ as a test function in \eqref{eq:LS-moment-equation} we have
\[M_0(t) = M_0(0) + \int_0^t \mathfrak n (u(s)) ds\] thus $M_0$ is monotonically increasing. The
result follows directly in the case $\Phi_{0}>0$ by \eqref{H:nucleation}. Therefore, we provide a
proof for $\Phi_{0}= 0$. We argue by contradiction. Suppose $M_0$ is bounded above independently of
time. By \eqref{eq:LS-moment-equation} with $\vphi(x)=x$ and noticing \eqref{eq:sublinear-rate} and
$0\leq u(t) \leq \rho$, we deduce
\[u'(t) = - u(t) \int_{0}^{\infty}a(x)f(t,x) + \int_{0}^{\infty} b(x) f(t,x) dx \in L^{\infty}_{t}(0,\infty)\,.\]
This entails $u\in W^{1,\infty}(0,\infty)$. Moreover, using Cauchy-Schwarz's inequality,
\[|u'(t)| \leq \left(\int_0^\infty a(x)f(t,x) dx\right)^{\tfrac 1 2} \left( \int_0^\infty
    a(x)(u(t)-\Phi(x))^2 f(t,x)dx\right)^{\tfrac 1 2} \,.\]
By \eqref{H:init} and Proposition
\ref{prop:lyap} with $k(x) = \int_{0}^{x}\Phi(y)dy$ -see also Remark \ref{rm:choices}, we have
\begin{equation}
\label{eq:dissip}
D_{k}(t) = \int_0^\infty a(x)(u(t)-\Phi(x))^2 f(t,x)dx \in L^1_{t}(0,\infty)\,.
\end{equation}
Thus, with \eqref{eq:sublinear-rate}, we deduce that $u' \in L^2(0,\infty)$. 

Next, we notice that $u^{k_0} \in L^1(0,\infty)$; this is due to \eqref{H:nucleation} and
\[ c\int_{0}^{\infty} u^{k_{0}}(t) dt \leq \int_{0}^{\infty} \mathfrak{n}(u(t))dt \leq  \limsup_{t\to\infty} M_{0}(t)-M_{0}(0),\]
together with the supposed bound on the $0^{th}$-order moment. Let
$p>\max(2,k_0)$, we have
\[ \int_0^\infty u^p(t) dt = \int_0^\infty u^{p-{k_0}}(t) u^{k_0}(t) \leq \rho^{p-k_0} \|u^{k_0}\|_{L^1(0,\infty)}\:,\]
thus $u^p \in L^1(0,\infty)$. Moreover $(u^p)' = p u' u^{p-1}$ belongs to $L^1(0,\infty)$ because
\[ \int_0^\infty |u' u^{p-1}| dt \leq \left(\int_0^\infty |u'|^2 \right)^{\tfrac 1 2} \int_0^\infty
  u^{2(p-1)} dt \leq \|u'\|_{L^2} \int_0^\infty u^{2(p-1)-k_0+k_0} dt\:. \]
But $2p -2 - k_0>0$, thus
$u^{2(p-1)-k_{0}}\leq \rho^{2(p-1)-k_{0}}$ and therefore $(u^p)'$ belongs to $L^1(0,\infty)$ because
$u^{k_{0}}$ does. This implies that $u^{p}\in W^{1,1}(0,\infty)$ and hence $u(t)^{p}\to 0$ as
$t\to +\infty$.

We now turn to the dissipation part \eqref{eq:dissip}. There exists a sequence of times
$t_n \to +\infty$ such that $D_{k}(t_n) \to 0$ by integrability. Actually the
dissipation reads
\[
D_{k}(t)= u^2(t)\int_0^\infty a(x)f(t,x)\, dx
+\int_0^\infty b(x) \Phi(x)f(t,x)\, dx-2u(t) \int_0^\infty b(x) f(t,x)\, dx\:.
\]
Using \eqref{eq:sublinear-rate} and the definition of solution, the first and last integrals are
continuous and bounded in time. Together with the fact that $u(t)\to 0$ as $t\to \infty$ we have
that
\[ 
\lim_{n\to +\infty} \int_0^\infty b(x) \Phi(x) f(t_n,x)dx = 0.
\]
Thanks to \eqref{eq:bphi} and Cauchy--Swartz's inequality we get
\[ \int_0^\infty x f(t,x) dx \leq \left(\int_0^\infty b(x)\Phi(x)f(t,x)dx\right)^{\tfrac 1 2} \left(
    \int_0^\infty C (1+x^2) f(t,x)dx\right)^{\tfrac 1 2}\:. \]
Since the $0^{th}$-order moment is
bounded and the second order moment as well by Corollary \ref{cor:lyap}, we obtain
\[\lim_{n\to+\infty} \int_0^\infty xf(t_n,x)dx = 0.\]
This contradicts the fact that $u(t_n) = \rho - \int_{0}^{\infty}xf(t_{n},x)dx\to 0$.
\end{proof}

Since mass is preserved, the divergence of $M_0$ implies that the average aggregate size tends to
zero. 

\subsection{Concentration behavior for the mass density}

We can use the dissipation to extract some information on the long-time asymptotic. {We proceed
  by  standart Lasalle's invariance principle arguments, proving that the orbits are relatively compact and
  we identify trajectories in the $\omega$-limit set to be time independent. In fact we shall work out the argument from scratch because we lack the continuity of our Lyapunov functional.}

In this section we assume our Hypotheses (\ref{H0}-\ref{H:nucleation}) hold true and we take $f$ a
solution in the sense of Definition \ref{def:sol} with initial data $f^{\rm in}$ satisfying \eqref{H:init}. We let $\{t_n\}$ an arbitrary
increasing sequence of times with $\lim_{n\to+\infty} t_n = +\infty$. Let $T>0$ arbitrary, we define for $t\in [0,T]$
\[ f^n(t,x) = f(t+t_n,x) \text{ and } \mu_t^n(dx) = x f^n(t,x) dx \,.\] The measures $\mu^n_t$ are
bounded nonnegative Radon measures on $(0,+\infty)$ with $\mu_t^n((0,+\infty)) = \rho$. Let
$H$ be given by \eqref{eq:H_2} and $H^n(t) = H(t+t_n)$. 

Let $\Mc^+_\rho(0,+\infty)$ the set of nonnegative Radon measures on $(0,+\infty)$ with mass less or
equal to $\rho$. The topology induced by the dual of $\Cc_c^0(0,\infty)$ on $\Mc^+_\rho(0,\infty)$
is called vague topology. We denote this space with such topology by $v-\Mc^+_\rho(0,\infty)$. This
space is metrizable and compact \cite{Brezis}, with metric
\[ d(\mu,\nu) = \sum_{k\geq 0} 2^{-k} \min ( 1 , |\langle \mu , \vphi_{k}\rangle-\langle \nu ,
  \vphi_{k} \rangle|) \]
for all $\mu$, $\nu$ in $\Mc^{+}_{\rho}(0,\infty)$ where
$(\vphi_{k})_{k\geq 0} \in \Cc_{c}^{\infty}(0,\infty)$ is dense in $\Cc_{c}^{0}(0,\infty)$ and $\langle \mu, \vphi\rangle=\int \vphi\, d\mu$ denotes
the duality pairing.

\begin{lemma}
  $\{\mu^n\}$ is relatively sequentially compact in $\Cc^0([0,T]), v-\Mc_{\rho}^+(0,\infty))$. 
\end{lemma}

\begin{proof}
Let $\vphi \in \Cc_c^\infty(0,\infty)$, we have by \eqref{eq:LS-moment-equation} that
\[\frac{d}{dt} \int_0^\infty \vphi(x) x f(t,x)dx = \int_0^\infty (x\vphi)' a(x) (u(t)-\Phi(x)) f(t,x) dx \,.\]
As $\vphi$ is compactly supported and $u$ is bounded by $\rho$, there exists a constant $C_\vphi>0$,
such that $|(x\vphi)'a(x)(u(t)-\Phi(x))|\leq C_\vphi x$, thus for all $t\geq 0$,
\[ \left| \frac{d}{dt} \int_0^\infty \vphi(x) x f(t,x)dx \right| \leq C_\vphi \rho \,.\]
Hence, for any $t_0\in[0,T]$
\[\limsup_{t\to t_0} \sup_n d(\mu^n(t),\mu^n(t_0)) \leq \limsup_{t\to t_0}  \sum_{k\geq 0} 2^{-k} \min(1, C_{\varphi_k}\rho |t-t_0|) = 0 \]
and thus the sequence is equicontinuous on $[0,T]$. We conclude by the
Arzelà-Ascoli theorem. 
\end{proof}

\begin{lemma}
  There exist $u\in \Cc^0([0,T])$ with $0\leq u \leq \rho$ and a subsequence of $\{u^n\}$ which
  converges to $u$ pointwise on $[0,T]$.
\end{lemma}

\begin{proof}
  Notice that by standart truncation arguments with \eqref{eq:LS-moment-equation},
  \eqref{eq:sublinear-rate} and \eqref{H:init},
  \[ \left| \frac{d}{dt} \int_0^\infty x^2 f(t,x) dx \right| \leq K(\rho+1) \int_0^\infty x(1+x)
    f(t,x) dx \,.\]
  The latter is uniformly bounded in time. By Arzelà-Ascoli we may extract a
  subsequence of $t\mapsto \int_{0}^\infty x^{2}f^{n}(t,x)dx$ which converges uniformly on
  $[0,T]$. Moreover, by the monotonicity and nonegativity of $H(t)$ it has a limit $H^{\infty}$ as $t\to \infty$
 and $H^n$ converges uniformly on $[0,T]$ to $H^\infty$. This proves that
  $\Psi(u^n(t)) = H^{n}(t) - \int_{0}^{\infty}x^{2}f^{n}(t,x)dx$ converges uniformly to a function
  $\overline \Psi(t)$ on $[0,T]$. Note that $\Psi = \int_{0}^{x} \Phi^{-1}(y)dy$ is continuous and
  strictly increasing (because $\Phi^{-1}>0$), thus $u^n(t)$ converges to
  $u(t) := \Psi^{-1}(\overline \Psi(t))$ pointwise.
\end{proof}

\begin{remark}
  We might replace $x^2$ above by $x^{1+\eta}$, provided that $(a(x)+b(x))x^\theta \leq C x$ for
  some $\theta \geq 0$ and $C>0$.
\end{remark}

\begin{prop}
$\mu_t$ converges to $0$ in $v-\Mc^+_\rho(0,+\infty)$ as $t\to+\infty$. 
\end{prop}

\begin{proof}
  By the previous two Lemmas we can extract a subsequence such that $\mu^n$ converges to some $\mu$
  in $\Cc^0([0,T]), v-\Mc_{\rho}^+(0,\infty))$ and $u^n$ converges poinwise to some
  $u \in\Cc^0([0,T])$. Given that $H^n(t)$ converges to $H^\infty$ for all $t\in[0,T]$, we use
  Corollary \ref{cor:lyap} and \eqref{eq:dotH_k} to deduce that
 \[ \lim_{n\to +\infty} \int_0^T \int_0^\infty \left(\Phi^{-1}(u^n(t))- \Phi^{-1}(\Phi(x))\right) \left(u^n(t)-\Phi(x)\right) a(x)f^n(t,x)dxdt = 0 \,.\]
    
Let $m\geq 1$ and $K_m = [\tfrac 1 m, m]$ and define $\chi_m$ a nonnegative continuous function with
compact support in $(0,\infty)$, equal to a positive constant on $K_m$ and such that
$\chi_m(x) \leq a(x)/x$ for all $x>0$. This is possible by the continuity and positivity of $a$. Hence,
\begin{equation}
\lim_{n\to+\infty}  \int_0^T \int_0^\infty \left(\Phi^{-1}(u^n(t)) -\Phi^{-1}(\Phi(x)) )\right)\left(u^n(t)-\Phi(x)\right)\chi_m(x) \mu^n_t(dx)dt =0 \,.
\end{equation}
Expanding the product, we easily see that space integrals converge uniformly in time on $[0,T]$. Moreover, $u^{n}$ converges pointwise and is bounded, hence the above limit can be interchanged with
the time integrals and we conclude that
\[ \int_0^T  \int_{K_{m}} \left(\Phi^{-1}( u(t)) -\Phi^{-1}(\Phi(x)) )\right)\left( u(t)-\Phi(x)\right)\chi_m(x)\mu_t(dx)dt =0 \]
for all $m\geq 1$. Thus, a.e. $t\in[0,T]$ we have the following measure equality 
\[\left(\Phi^{-1}(u(t)) - \Phi^{-1}(\Phi(x)) )\right)\left(u(t)-\Phi(x)\right) \mu_t(dx)  = 0 \,.\]
As $u$ is continuous, the above equality is achieved for all $t\in[0,T]$. Noticing that $\Phi^{-1}$
is increasing, for all $t\in[0,T]$, either $\mu_t = 0$ or both $u(t)>\Phi_{0}$ and there exists
$m_{t}:[0,T]\to (0,\rho]$ such that $\mu_t = m_{t} \delta_{\Phi^{-1}(u(t))}$.

We should prove the second alternative could not occur. Indeed, let $t_{0}$ such that
$\mu_{t_{0}} = m_{t_0}\delta_{x_{0}}$ with $x_{0}= \Phi^{-1}(u(t_{0})) >0$ -because
$u(t_{0})>\Phi_{0}$. For all $\vphi \in \Cc^{0}_{c}(0,\infty)$,
\[\lim_{x\to +\infty} \int_{0}^{\infty} \vphi(x) f^{n}(t_{0},x)dx = m_{t_0}\frac{\vphi(x_{0})}{x_{0}}\]
which proves that $f^{n}(t_{0},x)dx$ converges in $\Mc^{+}_{\rho}(0,\infty)$ and contradicts that
the $0^{th}$-moment $M_{0}(t_{0}+t_{n})$ goes to $\infty$ as $n\to\infty$, by Prop. \ref{prop:m0}.
\end{proof}

We conclude that the limit measure concentrates at the origin.
\begin{lemma}
\label{lm:origin}
There exists $m>0$ and a subsequence (not relabelled) such that 
\[ \lim_{n\to \infty} xf(t_n,x)dx = m \delta_0  \,, \qquad \text{weakly in } \Mc_\rho^+([0,+\infty)) \]
\end{lemma}

\begin{proof}
  Note that $\nu^n = xf(t_n,x)dx$ defines a sequence of measures on $[0,+\infty)$, such that
  $\nu^n_{|(0,+\infty)} = \mu^n$. The sequence $\{\nu^n\}$ is bounded and thus admits a subsequence
  which converges $v-\Mc^+_\rho([0,+\infty))$; recall that this is the dual of $\Cc^{0}_{c}([0,\infty))$. By the previous
  results, the limit $\nu$ verifies that $\nu_{|(0,+\infty)} = 0$ as measures. Thus
  $\nu = m \delta_0$ for some $m\ge 0$. We recall that $\langle \nu^n , x \rangle$ is uniformly
  bounded in $n$; this control allows to improve the convergence to
  $\langle \nu^n,\vphi \rangle \to \langle \nu , \vphi \rangle$ for all
  $\vphi \in\Cc_b([0,+\infty))$, the set of continuous and bounded functions.
\end{proof}

It remains to identify $m$ which will be the result of the next section.

\subsection{Identification of the concentrated mass}

This paragraph is devoted to the proof that $u(t)$ approaches the critical value $\Phi_{0}$; by mass
conservation, the concentrated mass $m$ in Lemma \ref{lm:origin} is therefore $\rho-\Phi_{0}$. This will conclude
the proof of Theorem \ref{thm:largetime}. In this section we still assume that our Hypotheses
(\ref{H0}-\ref{H:nucleation}) hold true and we take $f$ a solution in the sense of Definition \ref{def:sol} with initial datum $f^{\rm in}$
satisfying \eqref{H:init}.

\begin{lemma}
\label{lm:x^2}
We have that $\displaystyle \lim_{t\to +\infty} \int_0^\infty x^2 f(t,x) dx =0.$
\end{lemma}

\begin{proof}
  By \eqref{H:init} and a (refined) de La Vallé Poussin's lemma \cite{laurencotWeakCompactnessTechniques2015} there exists
  $\beta \in \Cc^1([0,\infty))$, nonegative, increasing, convex, such that
  $\lim_{x\to\infty}\beta(x)/x = +\infty$ and
\[\int_0^\infty \beta(x) x f^{\rm in}(x)dx < \infty \,.\]
According to Remark \ref{rm:choices} we may use the Lyapunov functional \eqref{eq:H_k} with
$k(x) = \beta(x) x$, which is convex too, to deduce that
\[ \sup_{t>0}\int_0^\infty \beta(x) x f(t,x)dx < \infty.\]
Let a sequence of times $t_n\nearrow \infty$; Lemma \ref{lm:origin} ensures that there exists a
subsequence (not relabelled) such that $xf(t_n,x) dx \to m \delta_0$ weakly. Thus, we let
$\chi_R = \min(x,R)$, so that
\begin{multline}\int_0^\infty x^2 f(t_n,x) dx \leq  \int_0^\infty \chi_R(x) x  f(t_n,x) dx + \int_R^\infty x^2 f(t_n,x)dx \\
\leq \int_0^\infty \chi_R(x) x  f(t_n,x) dx + \sup_{z>R} \left(\frac{z}{\beta(z)}\right) \sup_{t>0} \int_0^\infty \beta(x) x f(t,x)dx\:.
\end{multline}
We take the limit $n\to +\infty$; the first term on the right hand side goes to $0$ since
$xf(t_n,x)dx \to m\delta_0$ and $\chi_R(0)=0$. Finally we take the limit $R\to+\infty$ and the
remaining term goes to $0$. Since this is true for all sequences, we get the full limit.
\end{proof}

\begin{remark} Here again we might deal with $x^{1+\eta}$ instead of $x^2$. 
\end{remark}

\begin{lemma}
We have the following limit $\lim_{t\to +\infty} u(t) = \Phi_{0}$. 
\end{lemma}

\begin{proof}
  The convergence of $H(t)$ as $t\to \infty$ together with the previous lemma entails that
  $\Psi(u(t))$ converges to some constant $c$. By the (strict) monotonicity of $\Psi$ and its
  continuity, $\Psi^{-1}$ is also continuous and thus $u(t)$ converges to $\theta :=
  \Psi^{-1}(c)$. Assume that $\theta > \Phi_0$. We will prove a contradiction. We introduce the
  function $A(x) = \int_0^x \tfrac 1 a$, well-defined thanks to \eqref{H0} and \eqref{H:a}. Letting
  $A^n(x) = \int_{1/n}^x \tfrac 1 a$ for $x>1/n$ and $A^n(x) = 0$ otherwise, we get that $A^n$ is
  continuous and $A^n{}' \in L^\infty$. Thus we can use $A^n$ as a test function in the moment
  equation \eqref{eq:LS-moment-equation} and pass to the limit ($A^n\leq A$) to get
  \[ \frac{d}{dt} \int_0^\infty A(x) f(t,x) dx = u(t) M_0(t) - \int_0^\infty \Phi(x) f(t,x) dx\,.\]
  This is justified since the function $\Phi$ is continuous on $[0,\veps]$, thus bounded, and there
  exists $K_{\veps}>0$ such that $\Phi(x) =b(x)/a(x) \leq K_{\veps} x$ for all $x>\veps$ thanks to
  \eqref{H0} and \eqref{eq:sublinear-rate}. 
  
  As $\theta> \Phi_0$, we may find $\delta, t_0>0$ such that $u(t)>\Phi_0 +2\delta$ for all $t>t_0$. Next, we may find $\epsilon>0$ such that $\sup_{x\in[0,\veps]}\Phi(x) < \delta +\Phi_{0}$. Then 
  \[ \int_0^\infty \Phi(x) f(t,x) dx \leq \left(\delta+\Phi_{0}\right) M_{0}(t) + K_{\veps} \rho \,,\]
and collecting both estimates we arrive to, for all $t>t_0$,
\[ \frac{d}{dt} \int_0^\infty A(x) f(t,x) dx \geq \delta M_0(t) - K_{\veps} \rho\,.\]

Since $M_0(t) \to +\infty$ we derive that
\[ \lim_{t\to +\infty} \frac 1 t \int_0^\infty A(x) f(t,x) dx = + \infty.\]
But, with \eqref{H:a}, there exists $K_{A}$ such that $A(x)\leq K_{A}(1+x)$ for all $x>0$, thus
\[\int_0^\infty A(x) f(t,x) dx \leq K_A \int_0^\infty (1+x)f(t,x) dx \leq K_A M_0(t) + K_A\rho\:. \] 
Given that $M_0(t)/t$ is bounded we have a contradiction. Hence $u(t) \to \Phi_{0}$ as $t\to \infty$.
\end{proof}

As a consequence of the previous result and mass conservation, we have $m=\rho-\Phi_0$ in the
representation of the limit measure $\nu = m \delta_0$. This concludes the proof of Theorem
\ref{thm:largetime}.

\begin{corollary} Under the same hypotheses of the theorem, we have that all
  moments $M_\theta(t)$ with $\theta \in [0,1)$ diverge as $t\to\infty$.
\end{corollary}
This follows from Lemma \ref{lm:x^2}, by interpolating the first moment between moments of order $\theta$ and two.

\section{Long time behavior for power-law rates and numerics}
\label{sec:numeric}

\subsection{Rate of convergence for special cases}
The purpose of this section is to provide more detailed information of the long time behavior in specific power-law cases. Here we assume $a(x)=ax^\alpha $ and $b(x)=bx^\beta$ with $a,b>0$ and $0\le \alpha < \beta \le 1$. Then $\Phi(x)=\frac{b}{a} x^{\beta-\alpha}$ verifies $\Phi_0=0$ and is monotonically increasing.

We may define the classical moments of $f$ as
$$
M_k(t):= \int_0^\infty x^k f(t,x)\, dx, \quad k\in \mathbb{R}^+.
$$ 
In the case of power-law kinetic rates, the time derivative of the classical moments reads
    \[
    \frac{d M_k}{dt} = k a u(t)M_{k+\alpha-1}(t)-k b M_{k+\beta-1}(t)\:.
    \]


In what follows we consider some particular choices of the exponents $\alpha$ and $\beta$ for which more specific information can be given (this works also as a guide to numerical conjectures for the general case, see Section \ref{sec:simul} below).

\begin{lemma}[Case $\alpha=0$]
\label{lm:y}
Assume that $\mathfrak{n}(u)=u^{i_0}$ -a bound from above by a power law works the same way. 
Then there exists some $C>0$ such that, for advanced $t$

\[
M_0(t)\le C  t^{\frac{1}{1+i_0\beta}},\quad u(t)\le C  t^{-\frac{\beta}{1+i_0\beta}}\:.
\]
\end{lemma}
\begin{proof}
First we prove that $u(t)M_0^\beta(t)\le C$ for every $t\ge 0$. For that aim, 
let $y(t)=u(t)M_0^\beta(t)$. We readily compute
\[
y'=(b M_\beta-u a M_0)M_0^\beta + \beta u \mathfrak n(u) M_0^{\beta-1}\:.
\]
 Here we can use that moment interpolation yields the estimate 
\[
M_\beta(t)\le M_1^\beta(t) M_0^{1-\beta}(t)\le \rho^\beta \, M_0^{1-\beta}(t).
\]
Therefore, 
\[
y'\le  \varphi(t)+\, b\rho^\beta M_0-a uM_0 M_0^\beta = M_0(b\rho^\beta-a y)+\varphi(t)
\]
with $\varphi(t)\to 0$ as $t\to \infty$. This implies that $y(t)$ belongs to $[0,\rho^\beta b/a]$ for advanced times.

Thanks to our assumption on $\mathfrak{n}$ we have
\[
M_0'=\frac{u^{i_0}M_0^{\beta i_0}}{M_0^{\beta i_0}}\le \frac{\rho^{\beta i_0}(b/a)^{i_0}}{M_0^{\beta i_0}}
\]
and therefore
\[
\frac{d}{dt} M_0^{1+i_0\beta} \le (1+i_0\beta)\rho^{i_0\beta}(b/a)^{i_0}.
\]
Our statements follow easily from here.
\end{proof}

We expect to have equality in the previous estimates. Actually, in the general case $\alpha<\beta$ we conjecture that:
\begin{equation}\label{eq:conjecture}
M_0(t)\sim  t^{\frac{1}{1+i_0(\beta-\alpha)}},\quad u(t)\sim  t^{-\frac{\beta-\alpha}{1+i_0(\beta-\alpha)}}
\end{equation}
for advanced $t$. Likewise, we expect that $u(t)M_0^{\beta-\alpha}(t)$ will have a finite limit. Compare with the numerical simulations below.

We prove another partial result along the same lines.
\begin{lemma}[case $\beta=1$]
There holds that $$\lim_{t\to \infty} u(t)M_\alpha(t)=b\rho/a.$$
\end{lemma}
\begin{proof}
In this specific case we have
\[
\frac{dM_1}{dt}=a u M_\alpha - b M_1.
\]
This is integrated as
\[
M_1(t)=M_1(0)e^{-bt} +a e^{-bt} \int_0^t u(\tau) M_\alpha(\tau)e^{b\tau}\, d\tau
\]
Since $u(t)\to 0$ as $t\to \infty$, we have that $M_1\to \rho$  as $t\to \infty$. Now we argue by  contradiction. Assume that $\liminf_{t\to \infty} uM_\alpha>b\rho/a$. Then, given $\epsilon>0$ there exists some $T>0$ such that $uM_\alpha>b\rho/a+\epsilon$ for $t\ge T$. Therefore, for each $t>T$,
\[
M_1(t)\ge M_1(0)e^{-bt} +a e^{-bt} \int_0^T u(\tau) M_\alpha(\tau)e^{b\tau}\, d\tau
+a(b\rho/a+\epsilon)e^{-bt}\int_T^t e^{b\tau}\, d\tau 
\]
\[
=e^{-bt} \left(M_1(0)+a \int_0^T u(\tau) M_\alpha(\tau)e^{b\tau}\, d\tau\right)+e^{-bt}a (b\rho/a+\epsilon)\frac{e^{bt}-e^{bT}}{b}\:.
\]
Taking the limit $t\to\infty$ we obtain $\lim_{t\to \infty} M_1(t)\ge \rho+ a\epsilon/b$, which is a  contradiction. We can prove in a similar way that $\limsup_{t\to \infty} uM_\alpha<b\rho/a$ leads to a a contradiction. Thus our statement follows.
\end{proof}

\subsection{Numerical experiments and discussions}
\label{sec:simul}

To approach numerically the Lifshitz-Slyozov equation we use a standard finite volume scheme with an upwind approximation of the fluxes. The behaviour of the solutions is depicted in Figures \ref{fig:1} to \ref{fig:3}.


Figure \ref{fig:1} shows time evolution of the distribution for two distinct initial conditions with rates given by $\alpha = 1/3$ and $\beta = 2/3$, see details in the figure's legend. Note that we have no explicit solution at hand and the rate of convergence is unknown in this case. It seems that, roughly speaking, the particular details of the initial condition are lost as time advances
and the concentration behaviour that ensues seems to follow 
a universal profile. Figure \ref{fig:2} shows the rate of convergence of $u$, and divergence of $M_{0}$, to be polynomial. We compare with the conjecture \eqref{eq:conjecture} and the results agree. In fact, this is robust according to various set of coefficients (results not shown). To further capture the universal profile, we plot in Figure \ref{fig:3} the tail distribution $F(t,x)=\int_x^\infty f(t,x)dx$, while normalising the mass and the front speed. Specifically, choosing $x\mapsto 1/(1+M_0(t)) F(t,x/(1+M_0(t))$, we observe that several initial conditions lead to a similar asymptotic profile.

\section{Long-time behavior for constant $\Phi$}
\label{sec:const}

In this section we discuss the special case of $b(x)=\Phi_{0} a(x)$ for some given $\Phi_{0}>0$. Particular instances of this situation (with outflow behavior or zero boundary conditions) have been studied in \cite{Collet02,Collet04}. Here we investigate the case of nucleation boundary conditions \eqref{eq:bc}. For this section we
still assume \eqref{H0}, \eqref{H:ab_deriv} and \eqref{H:a}.  Note that  Hypothesis \eqref{H1} is
replaced by $\Phi$ being constant. Here \eqref{eq:bphi} is not required and \eqref{H:nucleation} is
replaced by the Lipschitz continuity of $\mathfrak n$ on $[\Phi_{0},\rho]$ and
\[ \mathfrak n (z) \geq \mathfrak n(\Phi_{0}) = 0\]
for all $z\in[0,\rho]$, that is, nucleation cannot fuel anymore at the critical value. Finally,
\eqref{H:init} reduces to $f^{\rm in} \in L^{1}(\Rb_{+},(1+x)dx)$ only. We shall assume morevover
that $a\in \Cc^{1}(0,\infty)$, which entails that characteristics curves are well-defined. Therefore, in this
case of constant $\Phi$,  existence and uniqueness of global solutions is ensured, see \cite{Calvo21}.

Further, we let 
\[A(x) = \int_{0}^{x} \frac 1 {a(y)} dy\] for all $x\geq 0$, which is an increasing
$\Cc^{1}$-diffeomorphism from $\Rb_{+}$ into $\Rb_{+}$ with $A(0)=0$. We denote
\[ M_{a}(t) = \int_{0}^{\infty}a(x)f(t,x) dx \]
which is finite for all $t\geq 0$ by \eqref{eq:sublinear-rate}.

The main result of this section rules out concentration phenomena for the density and provides an 
explicit rate of
convergence for $u$:
\begin{theorem}
\label{thm:alpha=beta}
Under the above hypotheses, any global solution in the sense of Definition \ref{def:sol} satisfies:
\begin{itemize}
\item \label{it:0_1} $\lim_{t\to\infty} u(t) = \Phi_{0}$,
\item \label{it:0_3} There exists $\overline f \in L^{1}(\Rb_{+},(1+x)dx)$ such that
  \[ \lim_{t\to+\infty} f(t,\cdot) = \overline f \,, \quad w-L^{1}(\Rb_{+})\,.\]
\end{itemize}	
\end{theorem}

Indeed, the limit $\overline f$ has a representation formula, given at end of the proof of Theorem \ref{thm:alpha=beta}, see
\eqref{eq:limrep}. This representation depends noticeably on the chosen initial condition. Furthermore, if $a$ is
non-decreasing in $x$ then $M_{a}$ is increasing in $t$ and our proof shows that the trend to equilibrium is exponential:
 \[ 0 < u(t) - \Phi_{0} \leq (u(0)-\Phi_{0}) e^{-M_{a}(0)t}\,.\]

\begin{proof} We proceed in a number of separate steps. During the proof we  assume $f^{\rm in} \neq
   0$. There is no loss of generality in so doing, as the given boundary conditions ensure that starting with $f^{\rm in}=0$ produces some nonvanising $f(t_0)$ for some $t_0>0$ small, which we may take as a new initial condition.

{Step 1: Mild formulation}. We will represent the solution in terms of
   characteristics. For that aim we will use several results from \cite{Calvo21}. The equation
   determining the characteristics reads
\begin{equation}
\label{eq:char1}
\frac d {ds} X(s;t,x) = (u(s)-\Phi_{0}) a(X(s;t,x))\,; \quad X(t;t,x)=x\,.
\end{equation}

For any given $x>0$ we can ensure existence and uniqueness of a maximal solution $X(\cdot; t, x)$ on
$(\sigma_{t}(x), \infty)$. Note the following: either $\sigma_{t}(x)=0$ and $\lim_{s\to 0} X(s,t,x)> 0$, or
$\sigma_{t}(x)>0$ and $\lim_{s\to \sigma_{t}(x)} X(s,t,x) = 0$.

Therefore, for any $s\in (\sigma_t(x),+\infty)$ we can integrate \eqref{eq:char1} as follows:
\[ X(s;t,x) = A^{-1}\left(A(x) +  \int_t^s (u(\tau)-\Phi_{0}) d\tau\right)\,.
\]
We define $\gamma(t)$ for all $t\geq 0$ through 
\begin{equation} \label{eq:gamma_alpha=beta}
\gamma(t) :=\int_0^t (u(\tau)-\Phi_{0}) d\tau\,, \text{ and } x_{c}(t) = A^{-1}(\gamma(t))\,.
\end{equation}
The curve $x_{c}(t)$ corresponds to $X(t;0,0)$. In \cite{Calvo21}, it is proved for all $t>0$,
$\sigma_{t}$ is a $\Cc^{1}$-diffeomorphism form $(0,x_{c}(t))$ into $(0,t)$ and $X(0;t,\cdot)$ is
also a diffeomorphism from $(x_{c}(t),\infty)$ into $(0,\infty)$. These facts provide the following mild formulation, for any bounded $\vphi \in \Cc^{0}([0,\infty))$
\begin{equation}
\label{eq:dual_nonrigoureux}
\int_0^\infty f(t,x)\vphi(x)dx = \int_0^t \mathfrak n (u(s)) \vphi(\sigma_t^{-1}(s))ds+\int_0^\infty f^{in}(x)\vphi(X(t;0,x))dx\,.
\end{equation}

{Step 2: $\gamma(t)$ is bounded}. Note that 
\[X(t;0,x) =  A^{-1}\left(A(x) + \gamma(t)\right) \geq A^{-1}(\gamma(t)) \]
Hence, by \eqref{eq:dual_nonrigoureux} with $\vphi(x)=x$,
\[ \rho \geq \int_{0}^{\infty}xf(t,x)dx \geq \int_0^\infty f^{in}(x) X(t;0,x) dx \geq A^{-1}(\gamma(t)) \int_{0}^{\infty}f^{\rm in}(x) dx \,, \]
which proves our claim.

{Step 3: $u(t) \to \Phi_{0}$ as $t\to \infty$}. Note that $\gamma'(t) = u(t)-\Phi_{0}$ and
\[ u'(t) = - (u(t)-\Phi_{0}) M_{a}(t)\,.\]
It follows that $\Phi_{0}$ is a steady state for $u(t)$. Since
$u(0)>\Phi_{0}$ and $M_{a}(t)$ is continuous and non-negative (with $M_{a}(t)=0$ if and only if $f(t)=0$, that is, if and only if
$u(t)=\rho$), we obtain that $u(t)$ decreases and $u(t)>\Phi_{0}$ for all $t\geq 0$. Thus $u$ has a limit,
which is $\Phi_{0}$.

{Step 4: $\gamma(t)$ is strictly increasing and has a limit as $t\to \infty$}. The previous steps ensures
$\gamma$ is (strictly) increasing, because $\gamma'(t)=u(t)-\Phi_{0}>0$, and $\gamma$ is bounded, so
$\gamma$ has a limit
\begin{equation}
\label{eq:limit-g}
\overline\gamma = \lim_{t\to\infty}\gamma(t) \,.
\end{equation}

{Step 5: $x_{c}(t)$ has a limit as $t\to \infty$}. Combining the definition of $x_{c}(t) = A^{-1}(\gamma(t))$, the
continuity of $A^{-1}$ and the limit of $\gamma$, we obtain that $x_{c}(t)$ has a limit,
\begin{equation}
\label{eq:limit-xc}
\overline x_{c} = \lim_{t\to \infty} x_{c}(t) =A^{-1}(\overline \gamma)\,.
\end{equation}

{Step 6: $X(t;0,\cdot)$ has a limit as $t\to \infty$}. We remark that $X(t;0,x) = A^{-1}(A(x) + \gamma(t))$ for
$x>0$ which entails that $X(t;0,x)$ has a limit too, since $A$ is $\Cc^{1}$-diffeomorphism, and
\begin{equation}
\label{eq:limit-X}
\overline X(x) = \lim_{t\to\infty} X(t;0,x) = A^{-1}(A(x) + \overline \gamma)\,.
\end{equation}
This makes clear that $\overline X$ is a $\Cc^{1}$-diffeomorphism from $(0,\infty)$ into
$(\overline x_{c},\infty)$.

{Step 7: $\sigma_{t}^{-1}$ has a limit as $t\to \infty$}. Recall from Step 1 that $\lim_{s\to \sigma_{t}(x)} X(s,t,x) = 0$. Using the fact that
$\lim_{x\to \sigma^{-1}_{t}(s)}X(s;t,x) = 0$ -see \cite{Calvo21}- we deduce
\[ \sigma_{t}^{-1}(s) = A^{-1}\left( \int_{s}^{t} (u(\tau) - \Phi_{0}) d\tau\right) = A^{-1}(\gamma(t) - \gamma(s))\,.\]
and thus we have the limit
\[ \overline{\sigma}^{-1}(s) = \lim_{t\to\infty} \sigma_{t}^{-1}(s) =A^{-1}(\overline \gamma -  \gamma(s))\,.\]
We conclude, $\overline \sigma(x) = \gamma^{-1}(\overline \gamma - A(x))$ is a
$\Cc^{1}$-diffeomorphism from $(0,\overline x_{c})$ into $(0,t)$, with reciprocal
$\overline{\sigma}^{-1}$.

{Step 8: The limit density}. Let $\vphi \in \Cc^{0}([0,\infty))$ be  bounded; we insert it in
\eqref{eq:dual_nonrigoureux}. By the  limit in Step 7 and the dominated convergence theorem, we conclude that
\begin{multline}\label{eq:limrep}
  \lim_{t\to\infty}\int_0^\infty f(t,x)\vphi(x)dx = \int_0^\infty \mathfrak n (u(s)) \vphi( \overline{\sigma}^{-1}(s))ds+\int_0^\infty f^{in}(x)\vphi(\overline X(x))dx \\
  = \int_0^{\overline x_{c}} \mathfrak n (u(\overline \sigma(x))) | \overline \sigma'(x)| \vphi(x) dx + \int_{\overline x_{c}}^\infty f^{in}(\overline{X}^{-1}(x))|\overline{X}^{-1}{}'|\vphi(x)dx 
\end{multline}
which proves the desired result.
\end{proof}

\bigskip

{\bf Acknowledgements:} J. C. is partially supported by the spanish  MINECO-Feder (grant RTI2018-098850-B-I00) and Junta de Andaluc\'ia (grants PY18-RT-2422 and  A-FQM-311-UGR18). E. H. and R. Y. have been supported by ECOS-Sud project n. C20E03 (France -- Chile) and acknowledge financial support from the Inria Associated team ANACONDA.

\bibliographystyle{abbrv}
\bibliography{biblio}

\begin{figure}[!h]
  \includegraphics[width=0.49\linewidth]{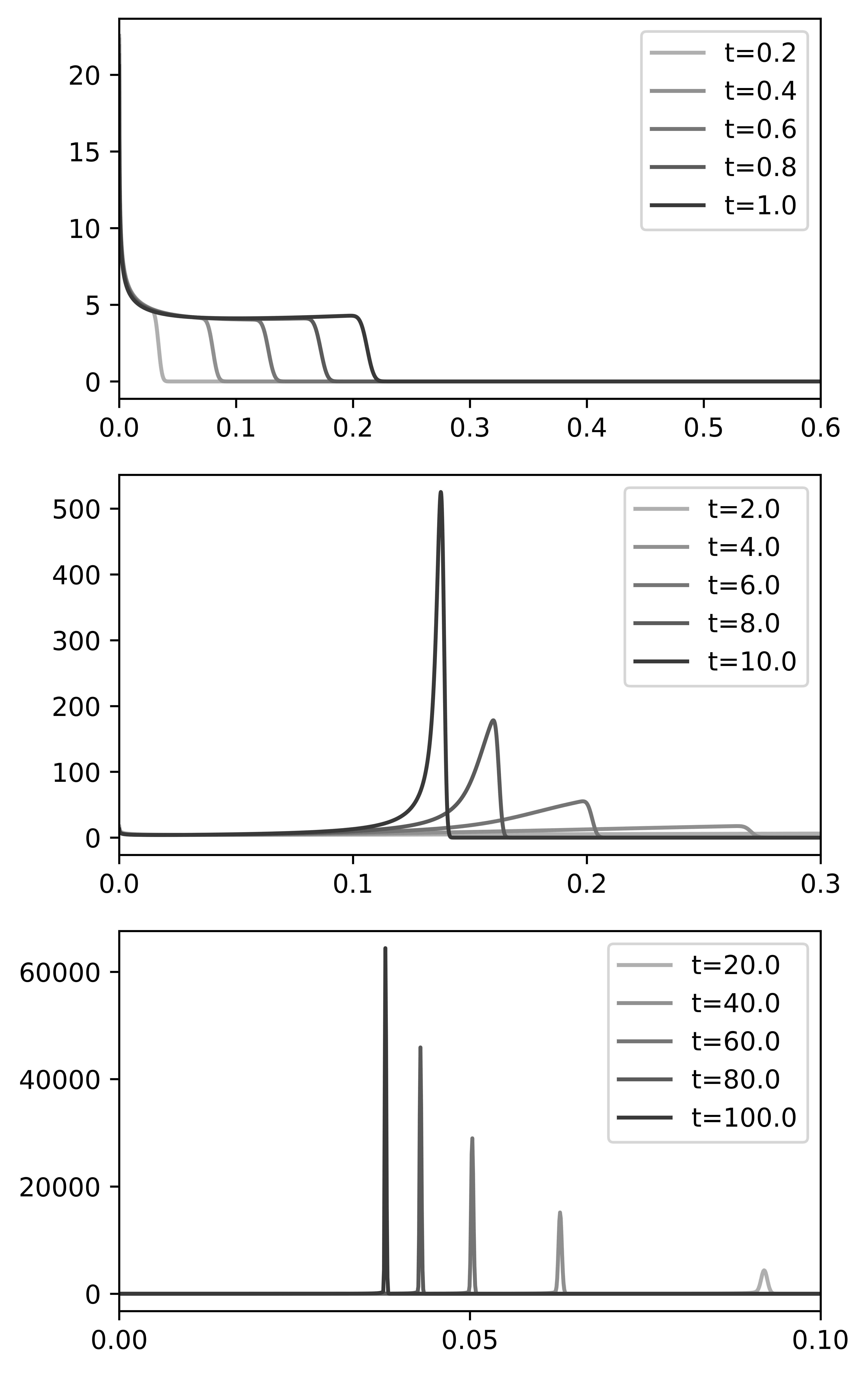}
  \includegraphics[width=0.49\linewidth]{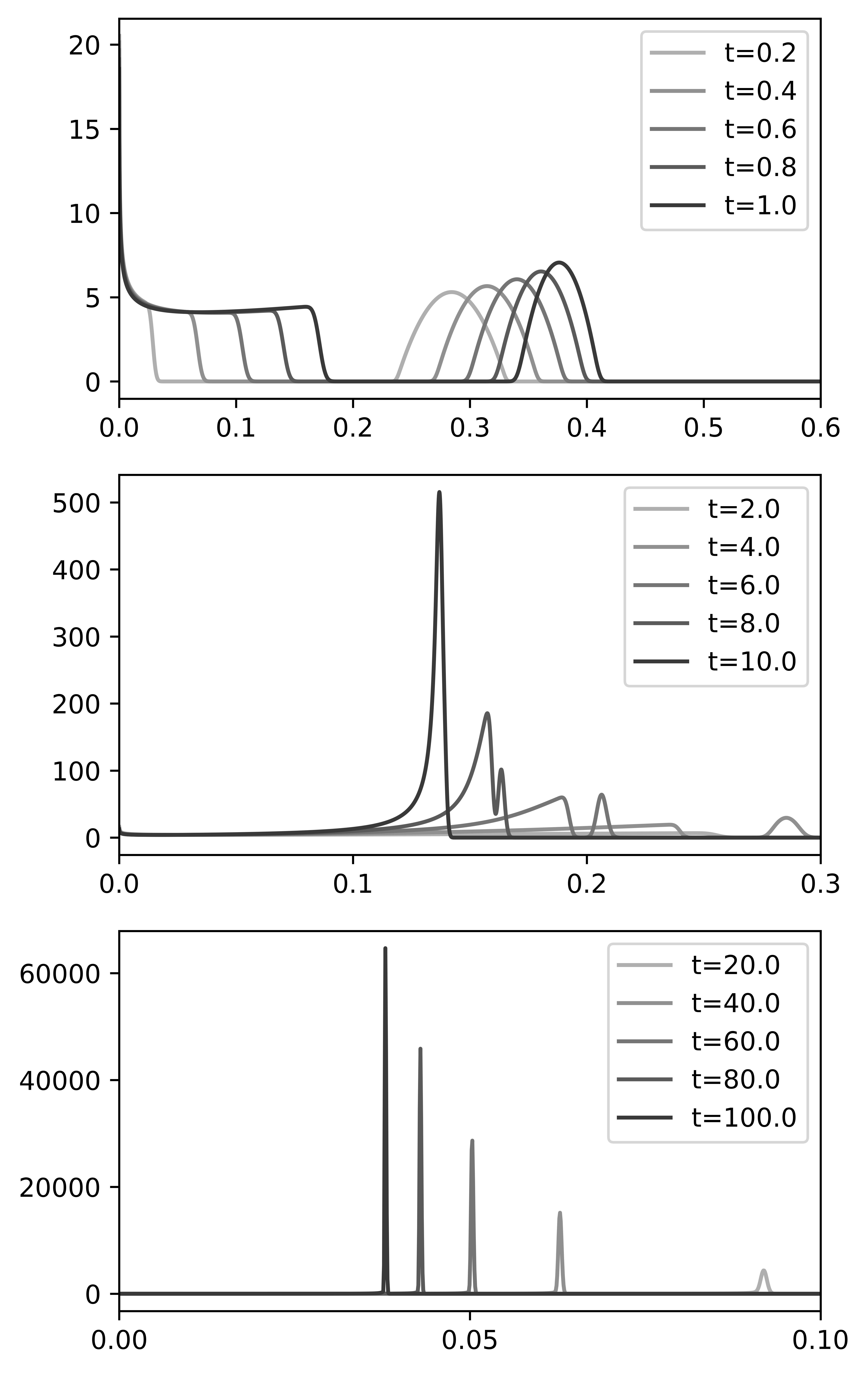}
 \caption{Distribution $f(t,x)$ with respect to $x$ at different times $t$. The rates are given by $a(x)=x^{1/3}$, $b(x) = x^{2/3}$ and $\mathfrak n(z)=z^{2}$. The mass total mass is $\rho=1$. Simulations were performed with a  finite volume scheme (upwind), with $\Delta t = 5\cdot 10^{-5}$ and $\Delta x = 10^{-4}$.  Left column: initial condition is $f^{\rm in}=0$; Right column: $f^{\rm in}(x) = (-2000(x-0.2)(x-0.3))_+$.}  \label{fig:1}
\end{figure}

\begin{figure}[!h]
    \centering
    \includegraphics[width=0.6\linewidth]{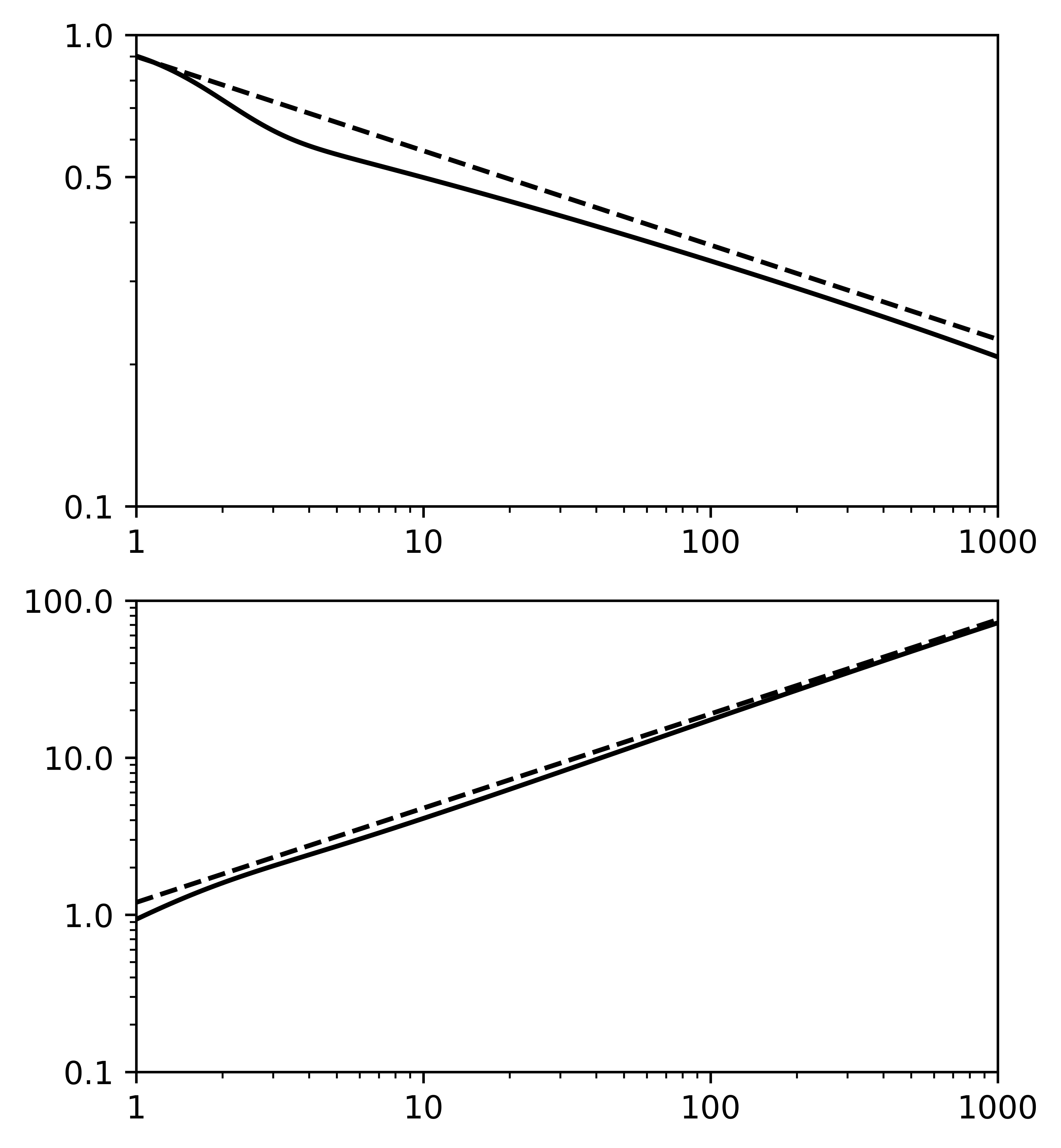}
    \caption{Up: $u(t)$ versus time $t$ in abscissa (straight line) and $t^{-1/5}$ (dashed line). Down: $M_{0}(t)$ versus time $t$ in abscissa (straight line) and $t^{3/5}$ (dashed line). Graphics are shown in log-log scale. Parameters are the same as Fig. \ref{fig:1} with $f^{\rm in}=0$.}   \label{fig:2}
\end{figure}

\begin{figure}[!h]
  \includegraphics[width=0.49\linewidth]{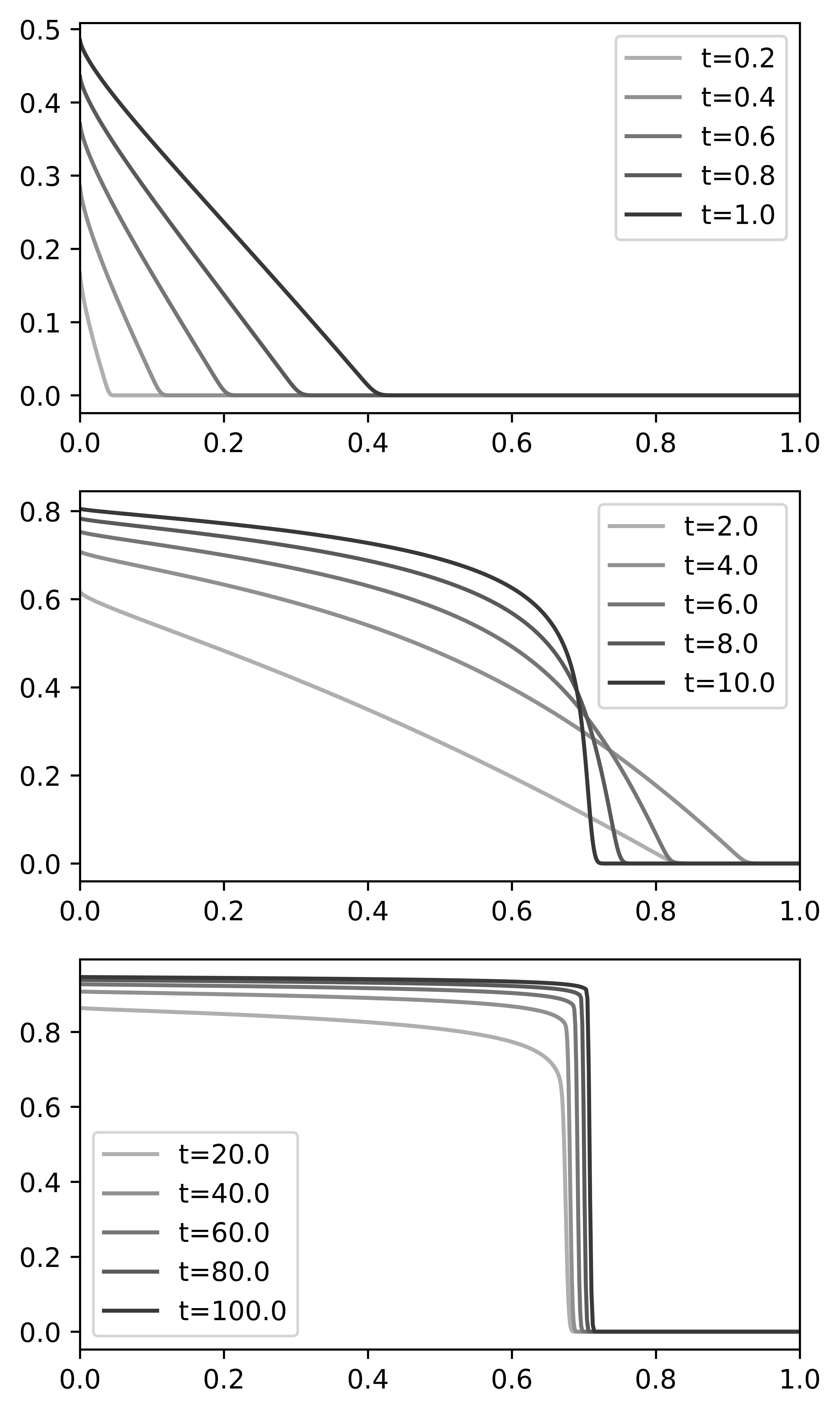}
  \includegraphics[width=0.49\linewidth]{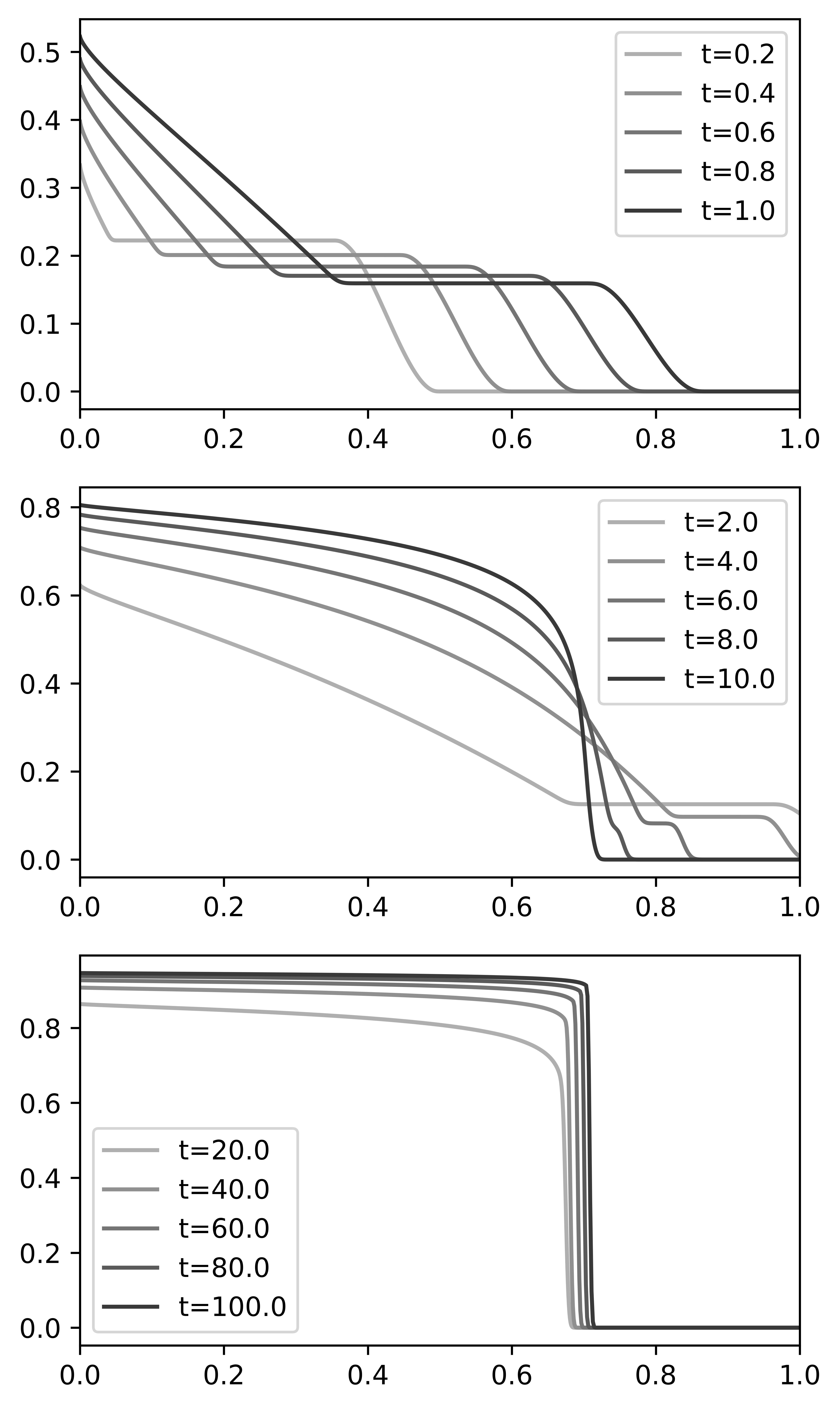}
 \caption{Normalised distribution $x\mapsto 1/(1+M_0(t)) F(t,x/(1+M_0(t))$ with $F(t,x)=\int_x^\infty f(t,x)dx$, with respect to $x$ at different times $t$. Parameters and simulations are performed as Fig. \ref{fig:1}.  Left column: initial condition is $f^{\rm in}=0$; Right column: $f^{\rm in}(x) = (-2000(x-0.2)(x-0.3))_+$.}\label{fig:3}
\end{figure}

\end{document}